\documentclass[reqno,oneside]{amsart}

\usepackage{hyperref}
\usepackage{amssymb,amsmath,amssymb,amsfonts,amsthm}
\usepackage{cases}
\usepackage{url}
\usepackage{geometry}
\usepackage{cite,mathtools}
\usepackage{xcolor}

\numberwithin{equation}{section}

\theoremstyle{plain}
\newtheorem{theorem}{Theorem}[section]
\newtheorem{lemma}[theorem]{Lemma}
\newtheorem{proposition}[theorem]{Proposition}
\newtheorem{corollary}[theorem]{Corollary}
\newtheorem{assume}[theorem]{Assumption}

\theoremstyle{definition}

\newtheorem{remark}[theorem]{Remark}

\allowdisplaybreaks[4]

\def\BE{\mathbb E}
\def\BN{\mathbb N}
\def\BR{\mathbb R}

\def\cE{\mathcal E}
\def\rd{\mathrm d}
\def\rdiv{\mathrm{div}}
\def\e{\mathrm e}
\def\supp{\mathrm{supp}}
\def\Ga{\Gamma}
\def\Om{\Omega}
\def\al{\alpha}

\def\ga{\gamma}
\def\de{\delta}
\def\ep{\epsilon}
\def\ve{\varepsilon}
\def\te{\theta}
\def\ze{\zeta}
\def\la{\lambda}
\def\si{\sigma}
\def\vp{\varphi}
\def\om{\omega}
\def\f{\frac}
\def\nb{\nabla}
\def\ov{\overline}
\def\pa{\partial}
\def\tri{\triangle}
\def\wt{\widetilde}

\title[Exponential stability for an infinite memory wave equation]
{Exponential stability for an infinite memory\\
wave equation with frictional damping\\
and logarithmic nonlinear terms}

\author[Qingqing Peng]{Qingqing Peng$^{1,2}$}
\thanks{$^1$School of Mathematics and Statistics \& Hubei Key Laboratory of Engineering Modeling and Scientific Computing, Huazhong University of Science and Technology, Wuhan 430074, China.}

\author[Yikan Liu]{Yikan Liu$^{2,*}$}
\thanks{$^2$Department of Mathematics, Kyoto University,
Kitashirakawa-Oiwakecho, Sakyo-ku, Kyoto 606-8502, Japan.\\
\indent$^*$Corresponding author.\\
\indent E-mail addresses: {\tt pengqq@hust.edu.cn} (Q. Peng), {\tt liu.yikan.8z@kyoto-u.ac.jp} (Y. Liu)}


\begin{document}

\maketitle

\begin{abstract}
This article is concerned with the energy decay of an infinite memory wave equation with a logarithmic nonlinear term and a frictional damping term. The problem is formulated in a bounded domain in $\BR^d$ ($d\ge3$) with a smooth boundary, on which we prescribe a mixed boundary condition of the Dirichlet and the acoustic types. We establish an exponential decay result for the energy with a general material density $\rho(x)$ under certain assumptions on the involved coefficients. The proof is based on a contradiction argument, the multiplier method and some microlocal analysis techniques. In addition, if $\rho(x)$ takes a special form, our result even holds without the damping effect, that is, the infinite memory effect alone is strong enough to guarantee the exponential stability of the system.
\end{abstract}

\section{Introduction}

Let $\Om\subset\BR^d$ ($d\in\BN:=\{1,2,\dots\}$) be a bounded and connected domain with a smooth boundary $\pa\Om$, which is divided into two parts as $\pa\Om=\Ga_0\cup\Ga_1$ such that $\Ga_0\cap\Ga_1=\emptyset$. The main purpose of this work is to investigate the energy decay of the following initial-boundary value problem for a nonlinear wave equation with infinite memory, frictional damping and logarithmic nonlinear terms:
\begin{equation}\label{1.1}
\left\{\begin{alignedat}{2}
& \begin{aligned}
& \rho(x)u_{tt}-\tri u+\int_0^\infty g(s)\,\rdiv(a(x)\nb u(t-s))\,\rd s\\
& \quad+b(x)h(u_t)=f(u):=|u|^{p-2}u\log|u|
\end{aligned} & \quad &\mbox{in }\Om\times(0,\infty),\\
& u=0 & \quad &\mbox{on }\Ga_0\times(0,\infty),\\
& \begin{aligned}
& \pa_\nu u-\int_0^\infty g(s)a(x)\pa_\nu u(t-s)\,\rd s=y_t,\\
& u_t+r(x)y_t+q(x)y=0
\end{aligned} & \quad &\mbox{on }\Ga_1\times(0,\infty),\\
& u(\,\cdot\,,-s)=u^0(\,\cdot\,,s),\ u_t(\,\cdot\,,0)=u^1 & \quad &\mbox{in }\Om,\ s\ge0,\\
& y=y^0 & \quad &\mbox{on }\Ga_1\times\{0\}.
\end{alignedat}\right.
\end{equation}
Here $\pa_\nu:=\nu\cdot\nb$ denotes the normal derivative and $\nu=\nu(x)$ is the unit outward normal vector at $x\in\pa\Om$. The function $\rho(x)$ represents the material density, and the term $b(x)h(u_t)$ stands for the possibly nonlinear damping effect. The nonlocal term
\begin{equation}\label{eq-nonlocal}
\int_0^\infty g(s)\,\rdiv(a(x)\nb u(t-s))\,\rd s
\end{equation}
reflects the memory effect from $t=-\infty$ to the current moment, where $g(t)$ represents the memory kernel. The appearance of \eqref{eq-nonlocal} usually models the dynamics of viscoelastic materials, which has gathered increasing popularity in recent years (see e.g.\! \cite{BDES18,KR22}). For the solvability of \eqref{1.1}, the initial displacement is given as a function $u^0$ defined in $\Om\times(-\infty,0]$ instead of merely in $\Om\times\{0\}$. The conditions on $\rho(x),b(x),h(s),g(t)$ and other coefficients involved in \eqref{1.1} will be specified later in Assumption \ref{assum1}. Throughout this article, we assume that the spatial dimensions $d$ and the exponent $p$ in the nonlinear term $f(u)$ satisfy
\begin{equation}\label{1.2}
d\ge3,\quad2<p<\f{2(d-1)}{d-2}.
\end{equation}
Finally, the boundary condition in \eqref{1.1} is a mixture of the homogeneous Dirichlet and the so-called acoustic boundary conditions.

The study of decay properties of solutions to wave equations with logarithmic nonlinearities has a long history due to their abundant applications in various research fields such as nuclear physics, optics and geophysics. In the absence of the nonlocal term \eqref{eq-nonlocal}, wave equations with logarithmic nonlinearities were introduced in Bialynicki-Birula and Mycielski  \cite{7,8}, which studied stable and localized solutions for $d=1$. Later, Cazenave and Haraux \cite{11} established the unique existence of solutions to Cauchy problems for $d=3$. In \cite{5}, Bartkowski and G\'orka obtained the existence of classical solutions and investigated weak solutions for the corresponding Cauchy problem for $d=1$. Further, G\'orka \cite{17} established the global existence of weak solutions to the initial value problem with $(u^0,u^1)\in H_0^1(\Om)\times L^2(\Om)$ for $d=1$ by using some compactness method. Recently, by the potential well method, the existence of global solutions to Cauchy problems was proved by Ye \cite{33}, where the exponential decay of global solutions was also obtained by introducing an appropriate Lyapunov function. Meanwhile, the blow-up of solutions in the unstable set was also verified. For more existing results on stability and blow-up of wave equations with logarithmic nonlinearities, see \cite{13,21,pe1,pe2} and the references therein.

Regarding the existence of both frictional damping and memory effects, there is also plenty literature. Cavalcanti and Oquendo \cite{c1} considered a model with two different dissipation effective only in a part of the domain, but the sum of the dissipation is effective in the whole domain. More precisely, they investigated
\[
u_{tt}-k_0\tri u+\int_0^t g(t-s)\,\rdiv(a(x)\nb u(s))\,\rd s+b(x)h(u_t)+f_1(u)=0
\]
and assumed the existence of a constant $\de>0$ such that $a+b\ge\de$ in $\Om$. Under this assumption, the authors showed the polynomial decay of the solution by assuming that the kernel $g$ decays polynomially. In Cavalcanti et al.\! \cite{c2}, the authors further considered viscoelastic wave equations with frictional damping on a compact Riemannian manifold with a boundary. Under the similar assumption on $a,b$ as above, they obtained decay results in a rather general setting with mild assumptions. Later, Cavalanti et al.\! \cite{c3} studied
\[
\rho(x)u_{tt}-\tri u+\int_0^\infty g(t-s)\,\rdiv(a(x)\nb u(s))\,\rd s+b(x)u_t=0,
\]
where $b\ge0$ acting on a region $A$ where $a=0$ in $A$. For a general material density $\rho$, the authors discussed the exponential decay of the system by combining microlocal analysis tools from Burq and G\'erard \cite{b1} and Koch and Tataru \cite{k1} when the kernel $g$ decays exponentially.  Moreover, if the material density is suitably chosen, it is even possible to remove the frictional damping $b(x)u_t$, that is, the nonlocal effect is strong enough to assure the exponential stability of the system. As relevant works on linear viscoelasticity with vector-valued solutions and anisotropic relaxation tensors, we refer to \cite{NO16,HLN23} where similar exponential decay was demonstrated. We also mention that viscoelastic systems are closely related to a class of time-fractional partial differential equations with polynomially decaying kernels $g$, whose long-time asymptotic behavior has also been studied extensively in the last decade (e.g.\! in \cite{xin,LLY15,HL23}).

On the other hand, almost all papers mentioned above are restricted to the Dirichlet boundary condition. Concerning wave equations with acoustic boundary conditions, Graber and Said-Houari \cite{g1} studied the semilinear problem with the porous acoustic boundary condition and proved the unique existence of local solutions by the semigroup theory. Imposing some restrictions on the source terms, they also proved that local solutions can be extended to global ones. In addition, stability and blow-up issues were also discussed. Liu and Sun \cite{36} considered the general decay for weak viscoelastic equations with acoustic boundary conditions under some conditions on the kernel $g$. For more stability results on wave equations with acoustic boundary conditions, see \cite{xin5,xin3}. In the context of acoustic boundary conditions, equations involving a damping term was first studied in Vicente and Frota \cite{v1} and later in Gao et al.\! \cite{g2}. In \cite{v1}, the authors investigated the uniform stabilization of nonlinear wave equations where $a(x)$ is positive in a neighborhood of a sub-boundary. Later, Gao et. al. \cite{g2} proved the uniform decay rates for damped wave equations with nonlinear acoustic boundary conditions by constructing an appropriate Lyapunov function, which generalized the results in \cite{v1}. Recently, Cavalcanti et al.\! \cite{c4} studied the local decay rates of the energy associated with a semilinear wave equation in an inhomogeneous medium with frictional damping and the acoustic boundary condition as follows:
\[
\begin{cases}
\rho(x)u_{tt}-\rdiv(K(x)\nb u)+a(x)h(u_t)+f_2(u)=0 &\mbox{in }\Om\times(0,\infty),\\
u=0 &\mbox{on }\Ga_0\times(0,\infty),\\
\begin{aligned}
& \pa_\nu^K u:=K(x)\nb u\cdot\nu=z_t,\\
& r_1(x)z_{tt}+r_2(x)z+g_1(z)+g_2(z_t)=-u_t
\end{aligned} &\mbox{on }\Ga_1\times(0,\infty),\\
u=u^0,\ u_t=u^1 &\mbox{in }\Om\times\{0\},\\
z=z^0 &\mbox{on }\Ga_1\times\{0\}.
\end{cases}
\]
Assuming that the damping is only active in a neighborhood of $\pa\Om$ and the material density $\rho$ is suitably chosen, the authors established the energy decay results by using microlocal analysis tools. What is more, for a general $\rho$, the authors proved the same result provided that the damping acts on a mesh of $\Om$.

Motivated by the above researches, in this article we are interested in the exponential stability of problem \eqref{1.1}. The main target of this paper is to generalize the result in \cite{c3,c4} in the following three directions.
\begin{enumerate}
\item The major nonlinearity in the governing equation of  \eqref{1.1} is logarithmic, which is technically more difficult than that in \cite{c3,c4}.
\item Compared with \cite{c3}, we consider the more general acoustic boundary condition. We construct auxiliary functions and employ the multiplier method to deal with the difficulties resulting from the acoustic boundary condition.
\item Compared with \cite{c4}, we include the infinite memory term into the governing equation. For a general material density $\rho(x)$, we can get rid of the damping region in a mesh like \cite{c4} and weaken the conditions on $b(x)$. In addition, the damping term $b(x)h(u_t)$ is removable if the $\rho$ is chosen suitably.
\end{enumerate}

The remainder of this paper is structured as follows. In Section 2, we prepare several lemmas including the local and global existence results of solutions as well as some conclusions about microlocal analysis. Then in Section 3, we state the main results concerning problem \eqref{1.1}. In Section 4, we provide the proof of the main theorem by using microlocal analysis tools, the multiplier method and a contradiction argument.

\section{Preliminary and Some Lemmas}

In this section, we fix notations and introduce some basic definitions, important lemmas and function spaces for the sake of the statements and proofs of our main results.

To simplify notations, we denote the norm of the Lebesgue space $L^p(\Om)$ ($1\le p\le\infty$) by $\|\cdot\|_p$, and the inner product of $L^2(\Om)$ by $(\,\cdot\,,\,\cdot\,)$. Similarly, we denote the norm and the inner product of $L^2(\Ga_1)$ by $\|\cdot\|_{\Ga_1}$ and $(\,\cdot\,,\,\cdot\,)_{\Ga_1}$, respectively. In the sequel, we abbreviate e.g.\! $u(x,t)=u(t)$, $v(x)=v$, $g(t)=g$ if there is no fear of confusion.

\begin{lemma}[see \cite{so1}]\label{lemma3}
Let $p\in[2,2_*]$ with $2_*=\f{2d}{d-2}$. Then there is an optimal constant $B_p>0$ such that
\[
\|v\|_p\le B_p\|\nb v\|_2,\quad\forall\,v\in H_{\Ga_0}^1(\Om),
\]
where
\[
H_{\Ga_0}^1(\Om):=\{v\in H^1(\Om)\mid v|_{\Ga_0}=0\}.
\]
\end{lemma}

We know that $H_{\Ga_0}^1(\Om)$ is a Sobolev space and we denote its norm by $\|\cdot\|_{\Ga_0}$. Next, we give some assumptions on the coefficients involved in problem \eqref{1.1}.

\begin{assume}\label{assum1}
The functions $\rho(x),a(x),g(t),b(x),r(x),q(x),h(s)$ satisfy the followings.
\begin{enumerate}
\item The function $\rho\in C^\infty(\Om)\cap C(\ov\Om)$ is strictly positive on $\ov\Om$.
\item The function $a\in C^\infty(\Om)\cap C(\ov\Om)$ is non-negative and $A:=\{x\in\Om\mid a(x)=0\}$ is a closed connected subset of $\Om$.
\item The function $g\in L^1(0,\infty)\cap C^1([0,\infty))$ is positive, non-increasing and satisfies
\[
\ell:=1-g_0\| a\|_\infty>0,\quad g_0:=\int_0^\infty g\,\rd t.
\]
Moreover, there exists a constant $\xi>0$ such that
\begin{equation}\label{xx2.2}
g'(t)\le-\xi\,g(t),\quad\forall\,t>0.
\end{equation}
\item The function $b\in L^\infty(\Om)$ is non-negative and there exists a constant $b_0>0$ such that
\[
b\ge b_0>0\quad\mbox{a.e.\! in }A.
\]
\item There exist constants $r_1>r_0>0$ and $q_1>q_0>0$ such that
\[
r_0\le r\le r_1,\quad q_0\le q\le q_1\quad\mbox{in }\Om.
\]
\item The function $h\in C^1(\BR)$ is non-decreasing and satisfies
\[
h(s)s>0,\quad\forall\,s\ne0.
\]
Moreover, there exist constants $h_2>h_1>0$ such that
\[
h_1s^2\le h(s)s\le h_2s^2,\quad\forall\,s\in\BR.
\]
\end{enumerate}
\end{assume}

\begin{remark}
The assumption \eqref{xx2.2} means that the decay rate of the kernel function $g$ is no less than exponential, for example, we can take $g(t)=\e^{-\xi t}$ with a positive constant $\xi$. Meanwhile, Assumption \ref{assum1}(vi) asserts that the function $h$ is approximately linear.
\end{remark}

Now we introduce several Hilbert spaces based on the functions $\rho,a,g$ satisfying Assumption \ref{assum1}. First, we introduce a $\rho$-weighted $L^2$ space
\[
L_\rho^2(\Om):=\left\{v:\Om\longrightarrow\BR\mid\int_\Om\rho|v|^2\,\rd x<\infty\right\}
\]
endowed with the norm
\[
\|v\|_\rho:=\left(\int_\Om\rho|v|^2\,\rd x\right)^{1/2}.
\]
Next, we define an $a$-weighted Sobolev space
\[
H_a^1(\Om):=\left\{v\in L^2(\Om)\mid\int_\Om a|\nb v|^2\,\rd x<\infty,\ v|_{\Ga_0}=0\right\}
\]
with the norm given by
\[
\|v\|_{H_a^1}:=\left(\int_\Om a|\nb v|^2\,\rd x\right)^{1/2}.
\]
Similarly, we define a $g$-weighted $L^2$ space
\[
L_g^2(\BR_+;H_a^1):=\left\{\eta:(0,\infty)\longrightarrow H_a^1(\Om)\mid\int_0^\infty g(t)\|\eta(t)\|_{H_a^1}^2\,\rd t<\infty\right\}
\]
endowed with the norm
\[
\|\eta\|_{L_g^2(\BR_+;H_a^1(\Om))}:=\left(\int_0^\infty g(t)\|\eta(t)\|_{H_a^1}^2\,\rd t\right)^{1/2}.
\]
In order to approach the existence result of problem \eqref{1.1}, we further define the following phase space
\[
\mathcal{H}:=H_{\Ga_0}^1(\Om)\times L_\rho^2(\Om)\times L_g^2(\BR_+;H_a^1)
\]
endowed with the norm
\[
\|(v,w,\eta)\|_{\mathcal{H}}:=\left(\|v\|_{\Ga_0}^2+\|w\|_\rho^2+\|\eta\|_{L_g^2(\BR_+;H_a^1)}^2\right)^{1/2}.
\]
Finally, let us introduce a function $\eta^t$ corresponding with the relative displacement history as
\[
\eta^t(x,s):=u(x,t)-u(x,t-s),\quad x\in\Om,\ t\ge0,\ s>0.
\]

In order to state our main result, we introduce the following energy functional associated with problem \eqref{1.1}:
\begin{align}
E(t) & :=\f12\int_\Om\rho|u_t(t)|^2\,\rd x+\f12\int_\Om k|\nb u(t)|^2\,\rd x+\f12\int_0^\infty g(s)\int_\Om a|\nb\eta^t(s)|^2\,\rd x\rd s\nonumber\\
& \quad\:\,-\int_\Om\f{|u(t)|^p\log|u(t)|}p\,\rd x+\int_\Om\f{|u(t)|^p}{p^2}\,\rd x+\f12\int_{\Ga_1}q|y(t)|^2\,\rd S,\label{x2.4}
\end{align}
where
\begin{equation}\label{eq-def-k}
k(x):=1-g_0a(x),\quad g_0:=\int_0^\infty g(t)\,\rd t.
\end{equation}

We first establish a lower estimate of $E(t)$.

\begin{lemma}\label{lemma 2.6}
Let Assumption $\ref{assum1}$ hold and $p$ satisfy \eqref{1.2}. We choose $\ve>0$ sufficiently small such that $p+\ve<2_*=\f{2d}{d-2}$ and define a function
\[
G(\la):=\f12\la^2-\f{B^{p+\ve}}{\e\,\ve p}\la^{p+\ve},\quad\la\ge0,
\]
where $B:=B_{p+\ve}\ell^{-1/2}$ and $B_{p+\ve}>0$ is the optimal constant defined in Lemma $\ref{lemma3}$. Then there holds $E(t)\ge G(\la(t))$ for any $t\ge0$ with $\la(t):=\sqrt\ell\,\|\nb u(t)\|_2$.
\end{lemma}

\begin{proof}
For any $v\in H_{\Ga_0}^1(\Om)$, we split $\Om$ into two parts according to the value of $v$ as
\begin{equation}\label{eq-Omega}
\Om_1:=\{x\in\Om\mid|v(x)|<1\},\quad\Om_2:=\{x\in\Om\mid|v(x)|\ge1\}.
\end{equation}
Then we employ the Sobolev embedding theorem and the definition of $B_{p+\ve}$ to estimate
\begin{align}
\int_\Om\f{|v|^p\log|v|}p\,\rd x & =\left(\int_{\Om_1}+\int_{\Om_2}\right)\f{|v|^p\log|v|}p\,\rd x\le\int_{\Om_2}\f{|v|^p\log|v|}p\,\rd x\nonumber\\
& \le\f1{\e\,\ve p}\int_{\Om_2}|v|^p|v|^\ve\,\rd x\le\f1{\e\,\ve p}\int_\Om|v|^{p+\ve}\,\rd x\le\f{B_{p+\ve}^{p+\ve}}{\e\,\ve p}\|\nb v\|_2^{p+\ve},\label{r3}
\end{align}
where we applied the inequality $\te^{-\ve}\log\te<(\e\,\ve)^{-1}$ for $\te\ge1$ with sufficiently small $\ve>0$ such that $p+\ve<\f{2d}{d-2}$.
Then we obtain
\begin{align*}
E(t) & \ge\f\ell2\int_\Om|\nb u(t)|^2\,\rd x+\f12\int_0^\infty g(s)\int_\Om a|\nb\eta^t(s)|^2\,\rd x\rd s-\f{B_{p+\ve}^{p+\ve}}{\e\,\ve p}\|\nb u(t)\|_2^{p+\ve}+\int_\Om\f{|u(t)|^p}{p^2}\,\rd x\\
& \ge\f\ell2\int_\Om|\nb u(t)|^2\,\rd x-\f{B^{p+\ve}}{\e\,\ve p}\left(\ell\|\nb u(t)\|_2^2\right)^{(p+\ve)/2}=\f12\la(t)^2-\f{B^{p+\ve}}{\e\,\ve p}\la(t)^{p+\ve}=G(\la(t)),
\end{align*}
which completes the proof.
\end{proof}

It is easy to verify that $G(\la)$ attains the maximum at
\[
\la_1=\left(\f{\e\,\ve p}{p+\ve}\right)^{\f1{p+\ve-2}}B^{-\f{p+\ve}{p+\ve-2}}>0
\]
with the maximum
\begin{equation}\label{5.2}
E_1=G(\la_1)=\left(\f12-\f1{p+\ve}\right)\la_1^2>0.
\end{equation}
With $\la_1$ and $E_1$, we recall the following result on an upper bound of $\la(t)$.

\begin{lemma}[see {\cite[Lemma 3.4]{23}}]\label{lemma 2.7}
Let Assumption $\ref{assum1}$ hold and $p$ satisfy \eqref{1.2}. Let $u$ be the solution to problem \eqref{1.1} with initial datum satisfying
\[
E(0)<E_1,\quad\la(0)<\la_1.
\]
Then there exists a constant $\la_2\in(0,\la_1)$ such that
\[
\la(t)=\sqrt\ell\,\|\nb u(t)\|_2\le\la_2,\quad\forall\,t\in [0,T_{\max}),
\]
where $T_{\max}$ is the maximum existence time.
\end{lemma}

Next, we give the existence theorem of solutions to problem \eqref{1.1}. Combining the Faedo-Galerkin method with the proof for the logarithmic nonlinearity in \cite{13} and that for nonlinear damping and memory terms in \cite{c4}, we can obtain

\begin{proposition}[Local existence]\label{theorem1}
Let $(u^0(\,\cdot\,,0),u^1,\eta^0)\in\mathcal{H}$ be given with $\eta^0(x,s)=u^0(x,0)-u^0(x,s)$. Let Assumption $\ref{assum1}$ hold and $p$ satisfy \eqref{1.2}. Then there exists $T>0$ such that problem \eqref{1.1} has a unique local weak solution $u$ in $\Om\times(0,T)$.
\end{proposition}

The proof of Theorem  \ref{theorem1} is typical, thus we omit the details.

The next lemma provides further estimates for the energy $E(t)$.

\begin{lemma}\label{lemma2.8}
Let all assumptions in Lemma $\ref{lemma 2.7}$ hold. Then for any $t\in [0,T_{\max}),$ there exists a constant
\begin{equation}\label{eq-def-C}
\wt C:=\f{2\la_2^{p+\ve-2}B^{p+\ve}}{\e\,\ve p}\left(1-\f{2\la_2^{p+\ve-2}B^{p+\ve}}{\e\,\ve p}\right)^{-1}
\end{equation}
such that
\begin{gather}
\int_\Om\f{|u(t)|^p\log|u(t)|}p\,\rd x\le\wt C E(t)\le\wt C E(0),\label{r1}\\
\BE(t)\le(1+\wt C)E(t)\le(1+\wt C)E(0),\label{r2}
\end{gather}
where
\begin{align}
\BE(t) & :=\f12\int_\Om\rho|u_t(t)|^2\,\rd x+\f12\int_\Om k|\nb u(t)|^2\,\rd x+\int_\Om\f{|u(t)|^p}{p^2}\,\rd x\nonumber\\
& \quad\:\,+\f12\int_0^\infty g(s)\int_\Om a|\nb\eta^t(s)|^2\,\rd x\rd s+\f12\int_{\Ga_1}q|y(t)|^2\,\rd S.\label{r4}
\end{align}
\end{lemma}

\begin{proof}
It directly follows from \eqref{r3} and Lemma \ref{lemma 2.7} that
\begin{align*}
\int_\Om\f{|u(t)|^p\log|u(t)|}p\,\rd x & \le\f{B_{p+\ve}^{p+\ve}}{\e\,\ve p}\|\nb u(t)\|_2^{p+\ve}\le\f{B^{p+\ve}}{\e\,\ve p}\la^{p+\ve-2}(t)\,\ell\|\nb u(t)\|_2^2\\
& \le\f{2\la_2^{p+\ve-2}B^{p+\ve}}{\e\,\ve p}\left(E(t)+\int_\Om\f{|u(t)|^p\log|u(t)|}p\,\rd x\right),
\end{align*}
which implies \eqref{r1}. For \eqref{r2}, we combine \eqref{r1} with \eqref{r4} to deduce
\[
\BE(t)\le E(t)+\int_\Om\f{|u(t)|^p\log|u(t)|}p\,\rd x\le(1+\wt C)E(t)\le(1+\wt C)E(0),
\]
which completes the proof.
\end{proof}

\begin{remark}
Clearly, the estimate \eqref{r2} indicates that $\BE(t)$ is uniformly bounded for all $t\in[0, T_{\max})$, which further implies that the solution is global in time and thus the maximum existence time $T_{\max}=\infty$. Besides, we have $0\le E(t)\le E(0)$ for all $t\in[0,\infty)$.
\end{remark}

Finally, we shortly introduce some conclusions about microlocal defect measures.

\begin{proposition}[{see \cite[Theorem 5.1]{c3}}]\label{the5.1}
Let $\{v_n\}_{n\in\BN}\subset L_{\mathrm{loc}}^2(\Om)$  be a bounded sequence which converges weakly to $0$ in $L_{\mathrm{loc}}^2(\Om)$. Then there exists a subsequence $\{v_{n'}\}\subset\{v_n\}$ and a positive Radon measure $\mu$ on $\Om\times S^{d-1}$ such that for all pseudo-differential operators $A$ of order $0$ on $\Om$ admitting a principal symbol $\si_0(A)\ge0$ and all $\chi\in C_0^\infty(\Om)$ such that $\chi\si_0(A)=\si_0(A),$ there holds
\begin{equation}\label{ss5.1}
(A(\chi v_{n'}),\chi v_{n'})\longrightarrow\int_{\Om\times S^{d-1}}\si_0(A)(x,\xi)\,\rd\mu(x,\xi),\quad\mbox{as }n'\to\infty.
\end{equation}
The above measure $\mu$ is called the microlocal defect measure of the subsequence $\{v_{n'}\}$.
\end{proposition}

\begin{remark}\label{rrk1}
Proposition \ref{the5.1} assures that all bounded sequences in $L_{\mathrm{loc}}^2(\Om)$ converging weakly to $0$ includes a subsequence which admits a microlocal defect measure. Taking $A=f\in C_0^\infty(\Om)$ in \eqref{ss5.1}, we easily observe that
\[
\int_{\Om\times S^{d-1}}f|v_{n'}|^2\,\rd x\longrightarrow\int_{\Om\times S^{d-1}}f\,\rd\mu(x,\xi)\quad\mbox{as }n'\to\infty.
\]
Therefore, $\{v_{n'}\}$ converges to $0$ if and only if $\mu=0$.

\end{remark}
For more definitions and conclusions about microlocal defect measures, refer to the appendices of \cite{c3,c4}.

\section{Statements of main results}

The aim of this section is to establish the exponential stability result concerning problem \eqref{1.1} as well as to investigate in what particular cases the system is exponentially stable in the absence of the frictional damping term.

In order to state our main result, we first propose the following assumption.

\begin{assume}\label{assum2}
For arbitrary fixed $T>0,$ $V\in L^\infty(0,T;L^d(\Om))$ and a nonempty open subset $\om\subset\Om,$ the unique solution $w\in C((0,T);L^2(\Om))\cap C([0,T];H^{-1}(\Om))$ to
\[
\left\{\begin{alignedat}{2}
& w_{tt}-\rdiv\left(\f k\rho\nb w\right)+V w=0 & \quad &\mbox{in }\Om\times(0,T],\\
& w=0 & \quad &\mbox{in }\om\times(0,T]
\end{alignedat}\right.
\]
is the trivial one $w\equiv0$. Here $k$ was defined in \eqref{eq-def-k}.
\end{assume}

The above assumption is concerned with the unique continuation property of hyperbolic equations, which has been investigated extensively in literature. As was described by Cavalcanti et al.\! \cite{c5}, in the case of $V=0$, Assumption \ref{assum2} is satisfied according to Burq and G\'erard \cite[(6.28)--(6.29)]{b1}. If $\f k\rho$ satisfies the assumptions of Theorem \ref{lem1} below (see also Duyckaerts, Zhang and Zuazua \cite{zua1}), then Assumptions \ref{assum2} is a consequence of the observability inequality e.g.\! in \cite[(2.15)]{zua1}. If $V\in L^\infty(0,T;L^d(\Om))$ and $\f k\rho=1$, then Assumption \ref{assum2} is locally fulfilled according to the pioneering work of Ruiz \cite{r1}. Moreover, in the more general case that $V\in L^{\f{d+1}2}(0,T;L^{\f{d+1}2}(\Om))$ and $\f k\rho$ is not necessarily $1$, it follows from Koch and Tararu \cite{k1} that the unique continuation also holds locally. In summary, Assumption \ref{assum2} holds true when $\f k\rho$ is in the situation of \cite{r1}, and it is locally satisfied in the context of \cite{k1,r1}. In a more general context, it should be assumed globally, such as that in \cite{c4,wang}.

Now we are well prepared to state the main result.

\begin{theorem}\label{lem1}
Let Assumptions $\ref{assum1},$ $\ref{assum2}$ and the conditions in Lemma $\ref{lemma2.8}$ hold. Let $E(t)$ be the energy defined by \eqref{x2.4} and assume that the initial energy
\begin{equation}\label{x3.4}
E(0)\le\min\left\{E_1, \left(\f{\ell\e\,\ve }{B_{p+\ve}^{p+\ve}}\right)^{\f1{p+\ve-2}}\f\ell{1+\wt C}\right\}:=R,
\end{equation}
where $E_1,\wt C$ are constants defined in $\eqref{5.2},\eqref{eq-def-C}$ respectively and $B_{p+\ve}$ is the optimal constant defined in Lemma $\ref{lemma3}$. Then there exist constants $T_0,C_0,\ga>0$ depending on the initial value such that
\begin{equation}\label{3.4}
E(t)\le C_0E(0)\,\e^{-\ga t},\quad\forall\,t>T_0.
\end{equation}
\end{theorem}

\begin{remark}\label{rk2}
Let us consider a special choice of $\rho(x)$ as
\begin{equation}\label{eq-rho}
\rho(x):=k(x)^{\f d{d-2}}=(1-g_0a(x))^{\f d{d-2}}.
\end{equation}
Since $a=0$ in $A$ by definition, we have $\rho=k=1$ in $A$, which means that both $\rho$ and $k$ are constant functions in the portion of $\Om$ where $a$ vanishes. Thus, any geodesic of the metric $(\f{K(x)}{\rho(x)})^{-1}$ is a straight line in the set $A$ with $K(x)=k(x)\delta_{ij}$, where $\delta_{ij}$ is the Kronecker delta. In other words, any geodesic in the metric $(\f{K(x)}{\rho(x)})^{-1}$ crossing $A$ cannot be trapped inside $A$. Moreover, \cite{c3} proved that all geodesics of the metric $(\f{K(x)}{\rho(x)})^{-1}$ reach the boundary $\pa\Om$. As $a>0$ in $\om:=\Om\setminus A$, the set $\om$ geometrically controls $\Om$,  namely, it satisfies the well-known geometric control condition. In other words, every geodesic of $\Om$ traveling with the unit speed and emitted at $t=0$ reach $\om$ in time $t<T_0$, for a certain time $T_0>0$. See \cite{c3} for further details about the above discussion.
\end{remark}  

According to Remark \ref{rk2} and the choice of the effective region of $a$, it is possible to stabilize the system only utilizing the viscoelastic effect. More precisely, in such a particular case, we are able to obtain the following exponential stability of the system without the damping term $b\,h(u_t)$ in \eqref{1.1}.

\begin{corollary}\label{co1}
Under the same assumption of Theorem $\ref{lem1},$ the same exponential stability result \eqref{3.4} still holds when the damping term $b\,h(u_t)$ is absent from problem \eqref{1.1}.
\end{corollary}

On the contrary, if $\rho(x)$ is simply a function satisfying Assumption \ref{assum1}(i), the presence of the damping term is essential.

\section{Proof of Theorem \ref{lem1}}

This section is devoted to the proof of Theorem \ref{lem1}. We mainly use the argument of contradiction to derive the exponential stability of the system, and we divide the proof into two cases according to the vanishing of $(u,\de)$, where $\de$ is an auxiliary function to be introduced later. In the case of $(u,\de)\ne(0,0)$, we construct suitable functions and apply Assumption \ref{assum2} to deduce the contradiction. In the case of $(u,\de)=(0,0)$, we use the method of multipliers and the microlocal tools.

To overcome a technical difficulty involving a passage to the limit in the boundary equation we work with an equivalent problem which has null initial data involving the boundary function.
For this purpose we define the auxiliary functions
\[
\vp(x,t):=y^0(x)+t\,\pa_\nu u^0(x),\quad\psi(x,t):=y(x,t)-\vp(x,t).
\]
For the definition of $\eta=\eta^t$ corresponding to the relative displacement history,
\[
\eta(x,t,s):=\eta^t(x,s)=u(x,t)-u(x,t-s),\quad x\in\Om,\ t\ge0,\ s\ge0.
\]
Then, proceeding formally, we obtain
\[
\eta_t+\eta_s=u_t,\quad\mbox{in }\Om\times(0,\infty)\times(0,\infty).
\]
Therefore, if $(u,y)$ is the solution of \eqref{1.1}, then the couple $(u,\psi)$ is a solution of
\begin{equation}\label{3.6}
\left\{\begin{alignedat}{2}
& \rho\,u_{tt}-\rdiv(k\nb u)+\int_0^\infty g(s)\rdiv(a\nb\eta^t(s))\,\rd s+b\,h(u_t)=f(u) & \quad &\mbox{in }\Om\times(0,\infty),\\
& \eta_t+\eta_s=u_t & \quad &\mbox{in }\Om\times(0,\infty)^2,\\
& u=0 & \quad &\mbox{on }\Ga_0\times(0,\infty),\\
& \begin{aligned}
& \pa_\nu^k u-\int_0^\infty g(s)a\pa_\nu\eta^t(s)\,\rd s=\vp_t+\psi_t,\\
& r(\vp_t+\psi_t)+q(\vp+\psi)=-u_t
\end{aligned} & \quad &\mbox{on }\Ga_1\times(0,\infty),\\
& u(\,\cdot\,,-s)=u^0(\,\cdot\,,s),\ u_t(\,\cdot\,,0)=u^1,\ \eta^0(\,\cdot\,,s)=u^0(\,\cdot\,,0)-u^0(\,\cdot\,,s) & \quad &\mbox{in }\Om,\ s\ge0,\\
& \eta^t(\,\cdot\,,0)=0, & \quad &\mbox{in }\Om,\ t\ge0,\\
& \psi=0 & \quad &\mbox{on }\Ga_1\times\{0\},
\end{alignedat}\right.
\end{equation}
where $\pa_\nu^k u:=k\nb u\cdot\nu$. We define the energy associated with \eqref{3.6} by
\begin{align*}
\cE(t) & :=\f12\int_\Om\rho|u_t(t)|^2\,\rd x+\f12\int_\Om k|\nb u(t)|^2\,\rd x+\f12\int_0^\infty g(s)\int_\Om a|\nb\eta^t(s)|^2\,\rd x\rd s\\
& \quad\:\,-\int_\Om\f{|u(t)|^p\log|u(t)|}p\,\rd x+\int_\Om\f{|u(t)|^p}{p^2}\,\rd x+\f12\int_{\Ga_1}q(\vp+\psi)^2(t)\,\rd S.
\end{align*}
Then it is not difficult to verify that
\begin{equation}\label{3.8}
\cE'(t)=-\int_\Om b\,h(u_t(t))u_t(t)\,\rd x+\f12\int_0^\infty g'(s)\int_\Om a|\nb\eta^t(s)|^2\,\rd x\rd s-\int_{\Ga_1}r(\vp+\psi)_t^2(t)\,\rd S
\end{equation}
for all $t\ge0$.  Then we have
\[
\cE(T)=\cE(0)-\int_0^T\!\!\!\int_\Om b\,h(u_t)u_t\,\rd x\rd t+\f12\int_0^T\!\!\!\int_0^\infty g'(s)\int_\Om a|\nb\eta^t(s)|^2\,\rd x\rd s\rd t-\int_0^T\!\!\!\int_{\Ga_1}r(\vp+\psi)_t^2\,\rd S\rd t
\]
for all $T\ge0$. Finally, we observe that $\cE(t)=E(t)$
for all $t\ge0$. The proof of theorem \ref{lem1} is equivalent to the proof of the following inequality: for all $T>T_0$ and $R>0$ verifying, there exists a constant $C=C(T,E(0))$ such that
\begin{align}
\cE(0) & \le C\left(\int_0^T\!\!\!\int_\Om b\,h(u_t)u_t\,\rd x\rd t-\int_0^T\!\!\!\int_0^\infty g'(s)\int_\Om a|\nb\eta^t(s)|^2\,\rd x\rd s\rd t\right.\nonumber\\
& \qquad\;\;\:\left.+\int_0^T\!\!\!\int_{\Ga_1}r(\vp+\psi)_t^2\,\rd S\rd t\right),\label{3.11}
\end{align}
where $(u,\psi,\eta)$ is the solution to problem \eqref{3.6}.

We are going to prove \eqref{3.11} for regular solutions and the result is obtained for mild solutions using standard density procedure. Our proof relies on contradiction arguments. Thus, in each case, there exist a time $T>T_0>0$ and a sequence of solutions $(u_n,\psi_n,\eta_n)$ to \eqref{3.6} verifying
\begin{equation}\label{3.13}
\cE_n(0)\le R
\end{equation}
for all $n\in\BN$, where
\begin{align*}
\cE_n(t) & :=\f12\int_\Om\rho|u_{n}'(t)|^2\,\rd x+\f12\int_\Om k|\nb u_n(t)|^2\,\rd x+\f12\int_0^\infty g(s)\int_\Om a|\nb\eta_n^t(s)|^2\,\rd x\rd s\nonumber\\
& \quad\;\:-\int_\Om\f{|u_n(t)|^p\log|u_n(t)|}p\,\rd x+\int_\Om\f{|u_n(t)|^p}{p^2}\,\rd x+\f12\int_{\Ga_1}q(\vp+\psi_n)^2(t)\,\rd S
\end{align*}
is the energy associated with the solution $(u_n,\psi_n,\eta_n)$ to \eqref{3.6} and we denote e.g.\! $u_t,u_{tt}$ by $u',u''$ etc.\! for simplicity. Meanwhile, there holds
\begin{align*}
\lim_{n\to\infty}\cE_n(0)\Bigg( &\int_0^T\!\!\!\int_\Om b\,h(u_n')u_n'\,\rd x\rd t-\int_0^T\!\!\!\int_0^\infty g'(s)\int_\Om a|\nb\eta_n^t(s)|^2\,\rd x\rd s\rd t\nonumber\\
& \left.+\int_0^T\!\!\!\int_{\Ga_1}r(\vp+\psi_n)_t^2\,\rd S\rd t\right)^{-1}=\infty.
\end{align*}
The above limit yields
\begin{align*}
\lim_{n\to\infty}\f1{\cE_n(0)}\Bigg( & \int_0^T\!\!\!\int_\Om b\,h(u_n')u_n'\,\rd x\rd t-\int_0^T\!\!\!\int_0^\infty g'(s)\int_\Om a|\nb\eta_n^t(s)|^2\,\rd x\rd s\rd t\nonumber\\
& \left.+\int_0^T\!\!\!\int_{\Ga_1}r(\vp+\psi_n)_t^2\,\rd S\rd t\right)=0,
\end{align*}
which, together with the estimate \eqref{3.13}, implies
\begin{align}
\lim_{n\to\infty}\Bigg( & \int_0^T\!\!\!\int_\Om b\,h(u_n')u_n'\,\rd x\rd t-\int_0^T\!\!\!\int_0^\infty g'(s)\int_\Om a|\nb\eta_n^t(s)|^2\,\rd x\rd s\rd t\nonumber\\
& \left.+\int_0^T\!\!\!\int_{\Ga_1}r(\vp+\psi_n)_t^2\,\rd S\rd t\right)=0.\label{3.18}
\end{align}
Besides form \eqref{xx2.2}, it follows that
\begin{align}
\lim_{n\to\infty}\Bigg( & \int_0^T\!\!\!\int_\Om b\,h(u_n')u_n'\,\rd x\rd t+\int_0^T\!\!\!\int_0^\infty g(s)\int_\Om a|\nb\eta_n^t(s)|^2\,\rd x\rd s\rd t\nonumber\\
& \left.+\int_0^T\!\!\!\int_{\Ga_1}r(\vp+\psi_n)_t^2\,\rd S\rd t\right)=0.\label{x3.18}
\end{align}
Combining \eqref{3.8} and \eqref{3.13}, we obtain $\cE_n(t)\le R$ for all $n\in\BN$ and all $t\ge0$. Then there exist weakly convergent subsequences of $\{u_n\}$ and $\{\vp+\psi_n\}$ (still denoted in the same way) and the respective weak limits $u:\Om\times(0,T)\longrightarrow\BR$ and $\de:\Ga_1\times(0,T)\longrightarrow\BR$ such that
\begin{alignat}{2}
u_n & \xrightharpoonup{\ {\rm w.}\ }u & \quad &\mbox{in }L^2(0,T;H_{\Ga_0}^1(\Om)),\label{3.20}\\
u_n' & \xrightharpoonup{\ {\rm w.}\ }u' & \quad &\mbox{in }L^2(0,T;L^2(\Om)),\label{3.21}\\
\vp+\psi_n & \xrightharpoonup{\ {\rm w.}\ }\de & \quad &\mbox{in }L^2(0,T;L^2(\Ga_1)),\label{3.22}\\
\vp'+\psi_n' & \xrightharpoonup{\ {\rm w.}\ }\de' & \quad &\mbox{in }L^2(0,T;L^2(\Ga_1))\label{3.23}
\end{alignat}
as $n\to\infty$. Since $H_{\Ga_0}^1(\Om)\hookrightarrow L^2(\Om)$ compactly, from the Rubin-Lions theorem \cite{14}, we have
\begin{equation}\label{3.24}
u_n\longrightarrow u\quad\mbox{in }L^2(0,T;L^2(\Om)),\ n\to\infty.
\end{equation}
On the other hand, since $2<p<\f{2d-2}{d-2}<2_*=\f{2d}{d-2}$, we can choose $\ve>0$ such that $2(p-1+\ve)<2_*$. Thus, for $f(u_n)=|u_n|^{p-2}u_n\log|u_n|$, we perform direct calculations and employ the Sobolev inequality to estimate
\begin{align}
\int_\Om|f(u_n)|^2\,\rd x & =\int_{\{|u_n|<1\}}|f(u_n)|^2\,\rd x+\int_{\{|u_n|\ge1\}}|f(u_n)|^2\,\rd x\nonumber\\
& \le(\e(p-1))^{-2}|\Om|+(\e\,\ve)^{-2}\int_{\{|u_n|\ge1\}}|u_n|^{2(p-1+\ve)}\,\rd x\nonumber\\
& \le(\e(p-1))^{-2}|\Om|+(\e\,\ve)^{-2}\left(B_{2(p-1+\ve)}\|\nb u_n\|_2\right)^{2(p-1+\ve)}\nonumber\\
& \le(\e(p-1))^{-p'}|\Om|+(\e\,\ve)^{-p'}\left(B_{2(p-1+\ve)}^2\f{2(1+\wt C)E_n(0)}\ell\right)^{p-1+\ve},\label{3.27}
\end{align}
where $B_{2(p-1+\ve)}$ is the optimal constant defined in Lemma \ref{lemma3}. Here we also used \eqref{r2} and the inequalities
\begin{gather*}
|\te^{p-1}\log\te|\le(\e(p-1))^{-1},\quad0<\te<1\\
\te^{-\ve}\log\te\le(\e\,\ve)^{-1},\quad\te\ge1,\ \ve>0.
\end{gather*}
Hence, from \eqref{3.24}, \eqref{3.27} and Lions Lemma (\cite[Lemma 1.3]{14}), we have
\[
f(u_n)\xrightharpoonup{\ {\rm w.}\ }f(u)\quad\mbox{in }L^2(0,\infty;L^2(\Om)),\ n\to\infty.
\]
From \eqref{3.18} and Assumption \ref{assum1}(v) we get
\begin{equation}\label{3.29}
\vp'+\psi_n'\longrightarrow0\quad\mbox{in }L^2(0,T;L^2(\Ga_1)),\ n\to\infty.
\end{equation}
At this moment, we used that $\psi_n(0)=0$. This allows us to conclude that
\begin{equation}\label{3.30}
(\vp+\psi_n)(0)\longrightarrow\vp(0)=y^0\quad\mbox{in }L^2(\Ga_1),\ n\to\infty.
\end{equation}
From \eqref{3.22} and \eqref{3.30}, we infer $\de(x,0)= y^0$ and hence
\[
(\vp+\psi_n)(t)=\int_0^t(\vp'+\psi_n')\,\rd s+(\vp+\psi_n)(0),\quad\de=\int_0^t\de'\,\rd s+\de(0)
\]
on $\Ga_1$ for $t\ge0$. Thus, using Cauchy-Schwarz inequality, we get
\begin{align}
& \quad\,\int_0^T\!\!\!\int_{\Ga_1}(\vp+\psi_n-\de)^2\,\rd S\rd t\nonumber\\
& \le2\int_0^T\!\!\!\int_{\Ga_1}\left(\int_0^t(\vp'+\psi_n'-\de')\,\rd s\right)^2\rd S\rd t+2\int_0^T\!\!\!\int_{\Ga_1}|(\vp+\psi_n)(0)-\de(0)|^2\,\rd S\rd t\nonumber\\
& \le2T^2\int_0^T\!\!\!\int_{\Ga_1}|\vp'+\psi_n'-\de'|^2\,\rd S\rd t+2T\int_{\Ga_1}|(\vp+\psi_n)(0)-y^0|^2\,\rd S.\label{3.34}
\end{align}
From \eqref{3.23} and \eqref{3.29}, we obtain that $\de'=0$. Thus, from \eqref{3.29}, \eqref{3.30} and \eqref{3.34} we obtain
\begin{equation}\label{3.35}
\vp+\psi_n\longrightarrow\de\quad\mbox{in }L^2(0,T;L^2(\Ga_1)),\ n\to\infty.
\end{equation}
At this point, to obtain the desired result, we divide the proof into two cases, namely, $(u,\de)\ne(0,0)$ and $(u,\de)=(0,0)$.

\subsection{The case of $(u,\de)\ne(0,0)$.}

For each $n\in\BN$, $(u_n,\eta_n,\psi_n)$ is a solution of the following problem
\begin{equation}\label{r13.36}
\left\{\begin{alignedat}{2}
& \rho\,u_n''-\rdiv(k\nb u_n)+\int_0^\infty g(s)\rdiv(a\nb\eta_n^t(s))\,\rd s+b\,h(u_n')=f(u_n) & \quad &\mbox{in }\Om\times(0,\infty),\\
& \eta_n'+\pa_s\eta_n=u_n' & \quad &\mbox{in }\Om\times(0,\infty)^2,\\
& u_n=0 & \quad &\mbox{on }\Ga_0\times(0,\infty),\\
& \begin{aligned}
& \pa_\nu^k u_n-\int_0^\infty g(s)a\,\pa_\nu\eta_n^t(s)\,\rd s=\vp'+\psi_n',\\
& r(\vp'+\psi_n')+q(\vp+\psi_n)=-u_n'
\end{aligned} & \quad &\mbox{on }\Ga_1\times(0,\infty),\\
& u_n(\,\cdot\,,-s)=u^0(\,\cdot\,,s),\ u_n'(\,\cdot\,,0)=u^1,\ \eta_n^0(\,\cdot\,,s)=u^0(\,\cdot\,,0)-u^0(\,\cdot\,,s) & \quad &\mbox{in }\Om,\ s\ge0,\\
& \eta_n^t(\,\cdot\,,0)=0, & \quad &\mbox{in }\Om,\ t\ge0,\\
& \psi_n=0 & \quad &\mbox{on }\Ga_1\times\{0\},
\end{alignedat}\right.
\end{equation}
Defining a sequence of auxiliary functions
\begin{equation}\label{la3}
z_n(t):=k\,u_n(t)+a\int_0^\infty g(s)\eta_n^t(s)\,\rd s,\quad n\in\BN,
\end{equation}
we take advantage of the second equation in \eqref{r13.36}, integration by parts and $\eta_n^t(\,\cdot\,,0)=0$ in $\Om$ for $t\ge0$ to deduce
\begin{align*}
z_n'(t) & =k\,u_n'(t)+a\int_0^\infty g(s)(\eta_n^t(s))'\,\rd s=k\,u_n'(t)+a\int_0^\infty g(s)(-\pa_s\eta_n^t(s)+u_n'(t))\,\rd s\\
& =u_n'(t)+a\int_0^\infty g'(s)\eta_n^t(s)\,\rd s.
\end{align*}
Let $\om'$ be an arbitrary closed subset of $\Om$ such that $A\subset\subset\om'$. Then there exists a constant $a_0>0$ depending on $\om'$ such that $a\ge a_0$ on $\ov{\Om\setminus\om'}$. Then we employ the Poincar\'e inequality to bound
\begin{equation}\label{la1}
\int_{\Om\setminus\om'}|\eta_n^t(s)|^2\,\rd x\le\f1{a_0\la_1}\int_{\Om\setminus\om'}a|\nb\eta_n^t(s)|^2\,\rd x,
\end{equation}
where $\la_1>0$ is the principal eigenvalue of the elliptic eigenvalue problem
\begin{equation}\label{eq-eigen}
\begin{cases}
-\rdiv(k\nb v)=\la\,v &\mbox{in }\Om,\\
v=0 &\mbox{on }\pa\Om.
\end{cases}
\end{equation}
Then from \eqref{x3.18}, \eqref{3.21} and \eqref{la1}, we yield
\begin{equation}\label{la2}
z_n'\xrightharpoonup{\ {\rm w.}\ }u'\quad\mbox{in }L^2(0,T;L^2(\Om\setminus\om')),\ n\to\infty.
\end{equation}
On the other hand, proceeding analogue computations and passing to the limit in \eqref{la3}, we deduce from \eqref{x3.18}, \eqref{3.24} and \eqref{la1} that
\[
z_n\longrightarrow k\,u\quad\mbox{in }L^2(0,T;L^2(\Om\setminus\om'))\hookrightarrow H^{-1}(0,T;L^2(\Om\setminus\om')),\ n\to\infty,
\]
where $H^{-1}(0,T)$ is the dual space of $H_0^1(0,T)$. Consequently, we obtain
\begin{equation}\label{la4}
z_n'\longrightarrow k\,u'\quad\mbox{in }H^{-1}(0,T;L^2(\Om\setminus\om')),\ n\to\infty.
\end{equation}
Then $k\,u'=u'$ in $L^2(0,T;L^2(\Om\setminus\om'))$ derive from convergence \eqref{la2} and \eqref{la4}. However, due to $k-1\ne0$ in $\Om\setminus\om'$ for any $\om'\supset\supset A$, there should hold
\[
u'\equiv0\quad\mbox{a.e.\! in }(\Om\setminus\om')\times(0,T),\ \forall\,\om'\supset\supset A'.
\]
Since $\om'$ and $T$ were chosen arbitrarily, we conclude
\begin{equation}\label{la5}
u'\equiv0\quad\mbox{a.e.\! in }\om\times(0,\infty),
\end{equation}
where we denote $\om:=\Om\setminus A$.

Passing $n\to\infty$ in the first, third and fourth equation in problem \eqref{r13.36} and utilizing the convergence \eqref{3.18}, \eqref{3.20}--\eqref{3.23} and \eqref{la5}, we arrive at
\begin{equation}\label{3.37}
\begin{cases}
\rho\,u''-\rdiv(k\nb u)=f(u) &\mbox{in }\Om\times(0,\infty),\\
u=0 &\mbox{on }\Ga_0\times(0,\infty),\\
\pa_\nu^k u=0 &\mbox{on }\Ga_1\times(0,\infty),\\
u'=0 &\mbox{in }\om\times(0,\infty).
\end{cases}
\end{equation}
Taking the $t$ derivative in the first equation in \eqref{3.37} and writing $w:=u'$ yield
\[
\begin{cases}
\rho\,w''-\rdiv(k\nb w)+V w=0 &\mbox{in }\Om\times(0,\infty),\\
w=0 &\mbox{in }\om\times(0,\infty),
\end{cases}
\]
where
\[
V:=-f'(u)=-(p-1)|u|^{p-2}\log|u|-|u|^{p-2}.
\]
Similarly to the estimates in \eqref{3.27}, one can show $V\in L^\infty(0,T;L^d(\Om))$ for any $T>0$. Then Assumption \ref{assum2} allows us to conclude that $u'=v\equiv0$ in $\Om\times(0,\infty)$ and consequently $u''\equiv0$ in $\Om\times(0,\infty)$.
Plugging this into \eqref{3.37} immediately gives
\begin{equation}\label{3.40}
\begin{cases}
-\rdiv(k\nb u)=f(u) &\mbox{in }\Om\times(0,\infty),\\
u=0 &\mbox{on }\Ga_0\times(0,\infty),\\
\pa_\nu^k u=0 &\mbox{on }\Ga_1\times(0,\infty).
\end{cases}
\end{equation}
Multiplying both sides of the first equation of \eqref{3.40} by $u$ and integrating over $\Om$, we use the divergence theorem to derive
\[
\int_\Om k|\nb u|^2\,\rd x=\int_\Om|u|^p\log|u|\,\rd x,
\]
which, together with \eqref{r3}, \eqref{r2} and \eqref{3.40}, indicates
\begin{align}
\ell\int_\Om|\nb u|^2\,\rd x & \le\int_\Om k|\nb u|^2\,\rd x=\int_\Om|u|^p\log|u|\,\rd x\le\f{B_{p+\ve}^{p+\ve}}{\e\,\ve }\|\nb u\|_2^{p+\ve}\nonumber\\
& \le\f{B_{p+\ve}^{p+\ve}}{\e\,\ve }\left(\f{(1+\wt C)E(0)}\ell\right)^{p-2+\ve}\|\nb u\|_2^2.\label{3.42}
\end{align}
Owing to the assumption \eqref{x3.4}, it turns out that
\[
\ell>\f{B_{p+\ve}^{p+\ve}}{\e\,\ve }\left(\f{(1+\wt C)E(0)}\ell\right)^{p-2+\ve}
\]
and eventually
\begin{equation}\label{3.43}
u=0\quad\mbox{in }\Om\times(0,\infty).
\end{equation}

On the other hand, we combine \eqref{3.29} and \eqref{3.35} to deduce $\de'=0$. Then passing $n\to\infty$ in the fifth equation of \eqref{r13.36} and using \eqref{3.43}, we conclude
\[
\de=0\quad\mbox{in }\Ga_1\times(0,\infty).
\]
Consequently, we arrive at $(u,\de)=(0,0)$, which contradicts with our original assumption.

\subsection{The case of $(u,\de)=(0,0)$.}

The proof for this case is more complicated and we divide this subsection into 5 steps.\medskip

{\bf Step 1.} In order to get the desired results, first we normalize
\[
\al_n=\sqrt{\cE_n(0)},\quad\wt u_n=\f{u_n}{\al_n},\quad\wt\eta_n^{\,t}=\f{\eta_n^t}{\al_n},\quad\wt\vp=\f\vp{\al_n},\quad\wt\psi_n=\f{\psi_n}{\al_n},\quad\wt u_n^{\,0}=\f{u^0}{\al_n},\quad\wt u_n^{\,1}=\f{u^1}{\al_n}.
\]
Thus, for each $n\in\BN$, $(\wt u_n,\wt\eta_n,\wt\psi_n)$ is the solution to
\begin{equation}\label{x13.36}
\left\{\begin{alignedat}{2}
& \rho\,\wt u_n^{\,\prime\prime}-\rdiv(k\nb\wt u_n)+\int_0^\infty g(s)\rdiv(a\nb\wt\eta_n^{\,t}(s))\,\rd s+b\f{h(\al_n\wt u_n^{\,\prime})}{\al_n}=\f{f(\al_n\wt u_n)}{\al_n} & \quad & \mbox{in }\Om\times(0,\infty),\\
& \wt\eta_n^{\,\prime}+\pa_s\wt\eta_n=\wt u_n^{\,\prime} & \quad & \mbox{in }\Om\times(0,\infty)^2,\\
& \wt u_n=0 & \quad & \mbox{on }\Ga_0\times(0,\infty),\\
& \begin{aligned}
& \pa_\nu^k\wt u_n-\int_0^\infty g(s)a\,\pa_\nu\wt\eta_n^{\,t}(s)\,\rd s=\wt\vp^{\,\prime}+\wt\psi_n^{\,\prime},\\
& r(\wt\vp^{\,\prime}+\wt\psi_n^{\,\prime})+q(\wt\vp+\wt\psi_n)=-\wt u_n^{\,\prime}
\end{aligned} & \quad & \mbox{in }\Ga_1\times(0,\infty),\\
& \wt u_n(\,\cdot\,,-s)=\wt u_n^{\,0}(\,\cdot\,,s),\ \wt u_n^{\,\prime}(\,\cdot\,,0)=\wt u_n^{\,1},\ \wt\eta_n^{\,0}(\,\cdot\,,s)=\wt u_n^{\,0}(\,\cdot\,,0)-\wt u_n^{\,0}(\,\cdot\,,s) & \quad & \mbox{in }\Om,\ s\ge0,\\
& \wt\eta_n^{\,t}(\,\cdot\,,0)=0 & \quad & \mbox{in }\Om,\ t\ge0,\\
& \wt\psi_n=0 & \quad & \mbox{in }\Ga_1\times\{0\}.
\end{alignedat}\right.
\end{equation}
The energy function associated with the above problem is given by
\begin{align}
\wt\cE_n(t) & :=\f12\int_\Om\rho|\wt u_n^{\,\prime}(t)|^2\,\rd x+\f12\int_\Om k|\nb\wt u_n|^2\,\rd x+\f12\int_0^\infty g(s)\int_\Om a|\nb\wt\eta_n^{\,t}(s)|^2\,\rd x\rd s\nonumber\\
& \quad\;\:-\int_\Om\f{|\al_n\wt u_n|^p\log|\al_n\wt u_n|}{p\al_n} \,\rd x+\int_\Om\f{|\al_n\wt u_n|^p}{p^2\al_n}\,\rd x+\f12\int_{\Ga_1}q(\wt\vp+\wt\psi_n)^2\,\rd S.\label{3.47}
\end{align}
We observe that $\wt\cE_n(t)\ge0$ from \eqref{3.42} for all $n\in\BN$ and all $t\in[0,T]$. Moreover, similarly to \eqref{3.8}, we have
\[
\wt\cE_n^{\,\prime}(t)=-\int_\Om b\f{h(\al_n\wt u_n^{\,\prime})}{\al_n}\wt u_n^{\,\prime}\,\rd x+\f12\int_0^\infty g'(s)\int_\Om a|\nb\wt\eta_n^{\,t}(s)|^2\,\rd x\rd s-\int_{\Ga_1}r(\wt\vp^{\,\prime}+\wt\psi_n^{\,\prime})^2\,\rd S\le0
\]
for all $t\ge0$. Thus,
\begin{align}
\wt\cE_n(0) & =\wt\cE_n(T)+\int_0^T\!\!\!\int_\Om b\f{h(\al_n\wt u_n^{\,\prime})}{\al_n}\wt u_n^{\,\prime}\,\rd x\rd t-\f12\int_0^T\!\!\!\int_0^\infty g'(s)\int_\Om a|\nb\wt\eta_n^{\,t}(s)|^2\,\rd x\rd s\rd t\nonumber\\
& \quad\,+\int_0^T\!\!\!\int_{\Ga_1}r(\wt\vp^{\,\prime}+\wt\psi_n^{\,\prime})^2\,\rd S\rd t.\label{3.49}
\end{align}
Similarly to \eqref{x3.18}, we get
\begin{align}
\lim_{n\to\infty}\Bigg( & \int_0^T\!\!\!\int_\Om b\f{h(\al_n\wt u_n^{\,\prime})}{\al_n}\wt u_n^{\,\prime}\,\rd x\rd t+\int_0^T\!\!\!\int_0^\infty g(s)\int_\Om a|\nb\wt\eta_n^{\,t}(s)|^2\,\rd x\rd s\rd t\nonumber\\
& \left.+\int_0^T\!\!\!\int_{\Ga_1}r(\wt\vp^{\,\prime}+\wt\psi_n^{\,\prime})^2\,\rd S\rd t\right)=0.\label{3.50}
\end{align}
Observing that $\wt\cE_n(0)=1$, \eqref{3.49} and \eqref{3.50}, we infer
\begin{equation}\label{3.51}
\lim_{n\to\infty}\wt\cE_n(T)=1.
\end{equation}

To finish the proof when $\rho$ is a general function, we shall show that the energy $\wt\cE_n(T)$ tends to $0$ uniformly, that is,
\[
\lim_{n\to\infty}\wt\cE_n(T)=0.
\]
Indeed, since $\wt\cE_n(t)\le\wt\cE_n(0)=1$, there exist weakly convergent subsequences of $\{\wt u_n\}$ and $\{\wt\vp+\wt\psi_n\}$ (still denoted in the same way) and the respective weak limits $\wt u$ and $\wt\de$ such that
\begin{equation}\label{3.53}
\begin{alignedat}{2}
\wt u_n & \xrightharpoonup{\ {\rm w.}\ }\wt u & \quad & \mbox{in }L^2(0,T;H_{\Ga_0}^1(\Om)),\\
\wt u_n^{\,\prime} & \xrightharpoonup{\ {\rm w.}\ }\wt u^{\,\prime} & \quad & \mbox{in }L^2(0,T;L^2(\Om)),\\
\wt\vp+\wt\psi_n & \xrightharpoonup{\ {\rm w.}\ }\wt\de & \quad & \mbox{in }L^2(0,T;L^2(\Ga_1)),\\
\wt\vp^{\,\prime}+\wt\psi_n^{\,\prime} & \xrightharpoonup{\ {\rm w.}\ }\wt\de^{\,\prime} & \quad & \mbox{in }L^2(0,T;L^2(\Ga_1))
\end{alignedat}
\end{equation}
as $n\to\infty$. Since $H_{\Ga_0}^1(\Om)\hookrightarrow L^2(\Om)$ compactly, it follows from the Aubin-Lions theorem that
\begin{equation}
\wt u_n\longrightarrow\wt u\quad\mbox{in }L^2(0,T;L^2(\Om)),\ n\to\infty.\label{3.57}
\end{equation}
Moreover, repeating the same argument used to prove \eqref{3.35}, we obtain
\begin{equation}\label{3.58}
\wt\vp+\wt\psi_n\longrightarrow\wt\de\quad\mbox{in }L^2(0,T;L^2(\Ga_1)),\ n\to\infty.
\end{equation}

{\bf Step 2.} Since $\{\al_n\}\subset[0,\infty)$ is bounded, there exists a constant $\al\ge0$ such that
\[
\lim_{n\to\infty}\al_n=\al
\]
possibly upon taking a further subsequence. We shall discuss two sub-cases of $\al>0$ and $\al=0$, and in both cases we aim to showing $\wt u=\wt\de=0$.\medskip

{\bf Case 1.} Let us assume $\al>0$. Since $(u,\de)=(0,0)$, \eqref{3.24} and \eqref{3.35} are equivalent with
\[
\begin{aligned}
\al_n\wt u_n & \longrightarrow0\quad\mbox{in }L^2(0,T;L^2(\Om)),\\
\al_n(\wt\vp+\wt\psi_n) & \longrightarrow0\quad\mbox{in }L^2(0,T;L^2(\Ga_1)),
\end{aligned}\quad n\to\infty
\]
and hence
\begin{align}
\al_n\wt u_n & \longrightarrow0\quad\mbox{a.e.\! in }\Om\times(0,T),\label{3.62}\\
\al_n(\wt\vp+\wt\psi_n) & \longrightarrow0\quad\mbox{a.e.\! in }\Ga_1\times(0,T)\nonumber
\end{align}
as $n\to\infty$. Due to the continuity of $|u|^{p-2}u\log|u|$, \eqref{3.62} implies
\[
f(\al_n\wt u_n)\longrightarrow0\quad\mbox{a.e.\! in }\Om\times(0,T),\ n\to\infty.
\]
Meanwhile, the estimate \eqref{3.27} guarantees the uniform boundedness of $\{|\al_n\wt u_n|^{p-2}\al_n\wt u_n\log|\al_n\wt u_n|\}$ in $L^\infty(0,T;L^2(\Om))$. Then we obtain
\begin{equation}\label{3.64}
f(\al_n\wt u_n)\xrightharpoonup{\ {\rm w.}\ }0\quad\mbox{in }L^2(0,T;L^2(\Om)),\ n\to\infty.
\end{equation}
Similarly to the proof of \eqref{la5}, again we can conclude
\[
\wt u^{\,\prime}\equiv0\quad\mbox{a.e.\! in }\om\times(0,T),\quad\om:=\Om\setminus A.
\]
Then passing $n\to\infty$ in \eqref{x13.36} results in
\begin{equation}\label{3.65}
\begin{cases}
\rho\,\wt u^{\,\prime\prime}-\rdiv(k\nb\wt u\,)=0 & \mbox{in }\Om\times(0,\infty),\\
\wt u=0 & \mbox{on }\Ga_0\times(0,\infty),\\
\pa_\nu^k\wt u=0 & \mbox{on }\Ga_1\times(0,\infty),\\
\wt u^{\,\prime}=0 & \mbox{in }\om\times(0,\infty),
\end{cases}
\end{equation}
from which we arrive at $\wt u=\wt\de=0$ following the same line as that in the case of $(u,\de)\ne(0,0)$.\medskip

{\bf Case 2.} Now it suffices to consider $\al=0$. We estimate
\[
\left|\f{f(\al_n\wt u_n)}{\al_n}\right|=\f{|\al_n\wt u_n|^{p-1}|\log|\al_n\wt u_n||}{\al_n}\le\al_n^{p-2}|\log\al_n||\wt u_n|^{p-1}+\al_n^{p-2}|\wt u_n|^{p-1}|\log|\wt u_n||.
\]
Since $p>2$ by \eqref{1.2}, it is readily seen that
\[
\lim_{n\to\infty}\al_n^{p-2}=\lim_{n\to\infty}\al_n^{p-2}|\log\al_n|=0,
\]
and both $|\wt u_n|^{p-1}$ and $|\wt u_n|^{p-1}|\log|\wt u_n||$ are uniformly bounded in $L^2(0,T;L^2(\Om))$, which indicates
\begin{equation}\label{x3.66}
\f{f(\al_n\wt u_n)}{\al_n}\longrightarrow0\quad\mbox{in }L^2(0,T;L^2(\Om)),\ n\to\infty.
\end{equation}
Therefore, passing $n\to\infty$ in \eqref{x13.36} again leads us to \eqref{3.65} and consequently $\wt u=\wt\de=0$.\medskip

As a result, it turns out that the weak limits $\wt u$ and $\wt\de$ vanish in both cases above, i.e.,
\[
\begin{aligned}
\wt u_n & \xrightharpoonup{\ {\rm w.}\ }0\quad\mbox{in }L^2(0,T;H_{\Ga_0}^1(\Om)),\\
\wt u_n^{\,\prime} & \xrightharpoonup{\ {\rm w.}\ }0\quad\mbox{in }L^2(0,T;L^2(\Om)),\\
\wt\vp+\wt\psi_n & \xrightharpoonup{\ {\rm w.}\ }0\quad\mbox{in }L^2(0,T;L^2(\Ga_1)),\\
\wt\vp^{\,\prime}+\wt\psi_n^{\,\prime} & \xrightharpoonup{\ {\rm w.}\ }0\quad\mbox{in }L^2(0,T;L^2(\Ga_1)),
\end{aligned}\quad n\to\infty.
\]
Moreover, \eqref{3.18} \eqref{3.57} and \eqref{3.58} improve these weak limits to the strong ones as
\[
\begin{aligned}
\wt u_n & \longrightarrow0\quad\mbox{in }L^2(0,T;L^2(\Om)),\\
\wt\vp+\wt\psi_n & \longrightarrow0\quad\mbox{in }L^2(0,T;L^2(\Ga_1)),\\
\wt\vp^{\,\prime}+\wt\psi_n^{\,\prime} & \longrightarrow0\quad\mbox{in }L^2(0,T;L^2(\Ga_1)),\\
\end{aligned}\quad n\to\infty.
\]

{\bf Step 3.} Pick $\ze\in C_0^\infty(0,T)$ and $\chi\in C_0^\infty(\Om)$ such that $\supp\,\chi\subset\Om\setminus A$. Since $a$ has a strictly positive lower bound in $\supp\,\chi$, we deduce
\begin{align*}
\int_\Om(|\chi|+|\nb\chi|)|\wt\eta_n^{\,t}|^2\,\rd x & =\int_{\supp\,\chi}(|\chi|+|\nb\chi|)|\wt\eta_n^{\,t}|^2\,\rd x\le\max_{\supp\,\chi}(|\chi|+|\nb\chi|)\int_{\supp\,\chi}|\wt\eta_n^{\,t}|^2\,\rd x\\
& \le C(\chi,\la_1)\int_{\supp\,\chi}|\nb\wt\eta_n^{\,t}|^2\,\rd x\le C(\chi,\la_1, a)\int_{\supp\,\chi}a|\nb\wt\eta_n^{\,t}|^2\,\rd x,
\end{align*}
where $\la_1$ was the principal eigenvalue of \eqref{eq-eigen}. Similar estimates as above will play an essential role in obtaining the convergence of the following expression. Let us set
\[
\ga_n(x,t):=\int_0^\infty g(s)\wt\eta_n^{\,t}(x,s)\,\rd s.
\]
Multiplying both sides of the first equation in \eqref{x13.36} by $\ga_n\,\chi\,\ze(t)$ and integrating over $\Om\times(0,T)$, we perform integration by parts to derive
\begin{equation}\label{3.75}
\sum_{j=0}^8I_j(n)=0,
\end{equation}
where
\begin{alignat*}{2}
I_0(n) & :=-\int_0^T\!\!\!\int_\Om\rho\,\wt u_n^{\,\prime}\ga_n'\,\chi\,\ze\,\rd x\rd t, & \quad I_1(n) & :=-\int_0^T\!\!\!\int_\Om\rho\,\wt u_n^{\,\prime}\ga_n\,\chi\,\ze'\,\rd x\rd t,\\
I_2(n) & :=\int_0^T\!\!\!\int_\Om k(\nb\wt u_n\cdot\nb\ga_n)\chi\,\ze\,\rd x\rd t, & \quad I_3(n) & :=\int_0^T\!\!\!\int_\Om k(\nb\wt u_n\cdot\nb\chi)\ga_n\,\ze\,\rd x\rd t,\\
I_4(n) & :=\int_0^T\!\!\!\int_0^\infty g\int_\Om a(\nb\wt\eta_n^{\,t}\cdot\nb\ga_n)\chi\,\ze\,\rd x\rd s\rd t, & \quad I_5(n) & :=\int_0^T\!\!\!\int_0^\infty g\int_\Om a(\nb\wt\eta_n^{\,t}\cdot\nb\chi)\ga_n\,\ze(t)\,\rd x\rd s\rd t,\\
I_6(n) & :=\int_0^T\!\!\!\int_\Om b\f{h(\al_n\wt u_n^{\,\prime})}{\al_n}\ga_n\,\chi\,\ze\,\rd x\rd t, & \quad I_7(n) & :=-\int_0^T\!\!\!\int_{\Ga_1}(\wt\vp^{\,\prime}+\wt\psi_n^{\,\prime})\ga_n\,\chi\,\ze\,\rd S\rd t,\\
I_8(n) & :=\int_0^T\!\!\!\int_\Om\f{f(\al_n\wt u_n)}{\al_n}\ga_n\,\chi\,\ze\,\rd x\rd t.
\end{alignat*}

Next, we prove that each term of \eqref{3.75} converges to $0$ as $n\to\infty$. Recalling the properties of the functions $a,\chi,g$, we have
\begin{equation}\label{3.76}
\left|\int_0^T\!\!\!\int_0^\infty g\int_\Om\wt\eta_n^{\,t}\,\chi\,\ze'\,\rd x\rd s\rd t\right|\le C(a,\ze,\chi,\la_1)\int_0^T\!\!\!\int_{\Om\setminus A}\int_0^\infty g\,\rd s\int_0^\infty g\,a|\nb\wt\eta_n^{\,t}|^2\,\rd s\rd x\rd t,
\end{equation}
which, together with \eqref{3.50}, indicates
\[
\lim_{n\to\infty}\int_0^T\!\!\!\int_0^\infty g\int_\Om\wt\eta_n^{\,t}\,\chi\,\ze'\,\rd x\rd s\rd t=0.
\]
Employing H\"older's inequality, \eqref{r2}, \eqref{3.47} and \eqref{3.76}, we deduce
\begin{align*}
|I_1(n)| & \le\int_0^T\!\!\!\int_\Om\rho|\wt u_n^{\,\prime}\ga_n\,\chi\,\ze'|\,\rd x\rd t\le C(\rho,\ze,\chi)\int_0^T\left(\int_\Om\rho|\wt u_n^{\,\prime}|^2\,\rd x\right)^{1/2}\left(\int_\Om\chi|\ga_n|^2\,\rd x\right)^{1/2}\rd t\\
& \le C(\rho,\ze,\chi)\cdot 2(1+\wt C)\sqrt{\wt\cE_n(0)}\int_0^T\left(\int_\Om\chi|\ga_n|^2\,\rd x\right)^{1/2}\rd t\\
& \le2(1+\wt C)\,C(\rho,\ze,\chi)\left(\int_0^T\!\!\!\int_\Om\chi|\ga_n|^2\,\rd x\rd t\right)^{1/2}\\
& \le2(1+\wt C)\,C(a,\rho,\ze,\chi)\left(\int_0^T\!\!\!\int_{\Om\setminus A}\int_0^\infty g\,\rd s\int_0^\infty g\,a|\nb\wt\eta_n^{\,t}|^2\,\rd s\rd x\rd t\right)^{1/2},
\end{align*}
which yield
\begin{equation}\label{x3.77}
\lim_{n\to\infty}I_1(n)=0.
\end{equation}
In the same manner, we conclude
\begin{equation}
\lim_{n\to\infty}I_2(n)=\lim_{n\to\infty}I_3(n)=0.
\end{equation}
Meanwhile, we take into account \eqref{3.50} and Assumption \ref{assum1} to see
\begin{align*}
& \quad\;\left|\int_0^T\!\!\!\int_0^\infty g\int_\Om b\f{h(\al_n\wt u_n^{\,\prime})}{\al_n}\wt\eta_n^{\,t}\,\ga_n\,\chi\,\ze\,\rd x\rd s\rd t\right|\\
& \le C(\ze,\chi,b)\int_0^T\!\!\!\int_0^\infty g\left(\int_\Om b\,\chi|h(\al_n\wt u_n^{\,\prime})|^2\,\rd x\right)^{1/2}\left(\int_\Om|\chi\,\wt\eta_n^{\,t}|^2\,\rd x\right)^{1/2}\rd s\rd t\\
& \le C(a,\ze,\chi,b,g,\al_n)\left(\int_0^T\!\!\!\int_\Om b\,\chi|\wt u_n^{\,\prime}|^2\,\rd x\rd t+\int_0^T\!\!\!\int_0^\infty g\int_\Om a|\nb\wt\eta_n^{\,t}|^2\,\rd x\right).
\end{align*}
As a consequence,
\begin{equation}
\lim_{n\to\infty}I_6(n)=0.
\end{equation}
In the same manner, we notice \eqref{x3.18}, \eqref{3.64} and \eqref{x3.66} to conclude
\begin{equation}\label{3.79}
\lim_{n\to\infty}I_4(n)=\lim_{n\to\infty}I_5(n)=\lim_{n\to\infty}I_7(n)=\lim_{n\to\infty}I_8(n)=0.
\end{equation}

It remains to deal with $I_0(n)$. Since $(\wt\eta_n^{\,t})'=-\pa_s\wt\eta_n^{\,t}+\wt u_n^{\,\prime}$, we have $I_0(n)=J_1(n)+J_2(n)$, where
\[
J_1(n):=-g_0\int_0^T\ze\int_\Om\chi\,\rho|\wt u_n^{\,\prime}|^2\,\rd x\rd t,\quad J_2(n):=-\int_0^T\ze\int_0^\infty g'(s)\int_\Om\chi\,\rho\,\wt\eta_n^{\,t}\,\wt u_n^{\,\prime}\,\rd x\rd s\rd t.
\]
It is not difficult to see $\lim_{n\to\infty}J_2(n)=0$. Therefore, we combine this with \eqref{3.75} and \eqref{x3.77}--\eqref{3.79} to acquire the desired convergence
\begin{align}
\lim_{n\to\infty}\int_0^T\!\!\!\int_\Om\chi\,\ze\,\rho|\wt u_n^{\,\prime}|^2\,\rd x\rd t=-\f1{g_0}\lim_{n\to\infty}J_1(n)=\f1{g_0}\lim_{n\to\infty}\left(J_2(n)+\sum_{j=1}^8I_j(n)\right)=0.\label{3.85}
\end{align}

On the other hand, let us return to the first equation in \eqref{x13.36} again to multiply $\ze\,\chi\,\wt u_n$ and repeat the same procedure as before to obtain
\begin{align*}
& -\int_0^T\ze\int_\Om\chi\,\rho|\wt u_n^{\,\prime}|^2\,\rd x\rd t-\int_0^T\ze'\int_\Om\chi\,\rho\wt u_n^{\,\prime}\wt u_n\,\rd x\rd t+\int_0^T\ze\int_\Om\chi\,k|\nb\wt u_n|^2\,\rd x\rd t\\
& +\int_0^T\ze\int_\Om k(\nb\wt u_n\cdot\nb\chi)\wt u_n\,\rd x\rd t+\int_0^T\ze\int_0^\infty g\int_\Om a(\nb\wt\eta_n^{\,t}\cdot\nb\chi)\wt u_n\,\rd x\rd s\rd t\\
& +\int_0^T\ze\int_\Om\chi\,a\left(\int_0^\infty g\nb\wt\eta_n^{\,t}\,\rd s\right)\cdot\nb\wt u_n\,\rd x\rd t+\int_0^T\ze\int_\Om\chi\,b\f{h(\al_n\wt u_n^{\,\prime})}{\al_n}\wt u_n\,\rd x\rd t\\
& -\int_0^T\ze\int_{\Ga_1}\chi(\wt\vp^{\,\prime}+\wt\psi_n^{\,\prime})\wt u_n\,\rd S\rd t-\int_0^T\ze\int_\Om\chi\f{f(\al_n\wt u_n)}{\al_n}\wt u_n\,\rd x\rd t=0.
\end{align*}
Mimicking the treatments for the previous case, we derive
\begin{equation}\label{3.86}
\lim_{n\to\infty}\int_0^T\int_\Om\ze\,\chi\,k|\nb\wt u_n|^2\,\rd x\rd t=0.
\end{equation}
From the convergence \eqref{3.85}--\eqref{3.86} and the definitions of $\ze$ and $\chi$, we arrive at
\[
\lim_{n\to\infty}\int_0^T\int_\Om\ze\,\chi\left(\rho|\wt u_n^{\,\prime}|^2+k|\nb\wt u_n|^2\right)\,\rd x\rd t=0.
\]
Observe that, we can choose $\ep>0$ arbitrarily small such that $0\le\ze\le1$, $\ze=1$ in $(\ep,T-\ep)$ and $\supp\,\ze\subset(0,T)$. Then the above limits yields
\[
\sqrt\rho\,\wt u_n^{\,\prime},\sqrt k\,\nb\wt u_n\longrightarrow0\quad\mbox{in }L^2(0,T;L^2(\Om\setminus A)),\ n\to\infty.
\]

{\bf Step 4.} Let $\mu$ be the microlocal defect measure associated with $\{\wt u_n\}$ in $H^1((0,T)\times(\Om\setminus A))$. The above convergence and Remark \ref{rrk1} imply that $\mu=0$ in $(0,T)\times(\Om\setminus A)$, that is, $\supp\,\mu\subset(0,T)\times A$.

On the other hand, by the above convergence, $a=0$ in $A$ and \eqref{x13.36}, we find
\[
\square\wt u_n=:\rho\wt u_n^{\,\prime\prime}-\rdiv(k\nb\wt u_n)=-\f{f(\al_n\wt u_n)}{\al_n}-b\f{h(\al_n\wt u_n^{\,\prime})}{\al_n}\quad\mbox{in }A\times(0,T),\quad\forall\,n\in\BN.
\]
Using \eqref{3.50}, \eqref{3.64} and \eqref{x3.66}, we obtain
\[
\square\wt u_n\longrightarrow0\quad\mbox{in }\quad H^{-1}(\Om\times(0,T)),\ n\to\infty.
\]
and hence
\[
\square\wt u_n^{\,\prime},\square(\nb\wt u_n)\longrightarrow0\quad\mbox{in }H^{-2}(\Om\times(0,T)),\ n\to\infty,
\]
which is equivalent with
\[
\supp\,\mu\subset\left\{(t,x,\tau,\xi)\mid\tau^2=\f{k(x)}{\rho(x)}\|\xi\|^2\right\}.
\]
In addition, $\supp\,\mu$ is the union of curves which are the bicharacteristics of the principal symbol
\[
q=\tau^2-\f{k(x)}{\rho(x)}\|\xi\|^2.
\]

Since $T>T_0$, each bicharacteristic ray enters the region $\Om\setminus A$ before the time $T$. Then we obtain $\mu=0$ in $\Om$ and  consequently
\begin{equation}\label{3.92}
\wt u_n^{\,\prime},\nb\wt u_n\longrightarrow0\quad\mbox{in }L^2(\Om\times(0,T)).
\end{equation}

{\bf Step 5.} Now we are in a position to deal with the logarithmic nonlinearity as well as the boundary term. Splitting $\Om$ into $\Om_1,\Om_2$ in the same way as that in \eqref{eq-Omega} and using the Sobolev embedding theorem, we estimate
\begin{align}
\left|\int_\Om\f{|\al_n\wt u_n|^p\log |\al_n\wt u_n|}{\al_np} \,\rd x\right| & \le\left(\int_{\Om_1}+\int_{\Om_2}\right)\f{|\al_n\wt u_n|^p|\log|\al_n\wt u_n||}{\al_n p}\,\rd x\nonumber\\
& \le\f{(\e(p-2))^{-1}\al_n}p\|\wt u_n\|_2^2+\f{(\e\,\ve)^{-1}\al_n^{p-1+\ve}}p\|\wt u_n\|_{p+\ve}^{p+\ve}\nonumber\\
& \le\f{(\e(p-2))^{-1}\al_n B_2^2}p\|\nb\wt u_n\|_2^2+\f{(\e\,\ve)^{-1}\al_n^{p-1+\ve}B_{p+\ve}^{p+\ve}}p\|\nb\wt u_n\|_2^{p+\ve}\nonumber\\
& \le\xi_1\|\nb\wt u_n\|_2^2+\xi_2\|\nb\wt u_n\|_2^2,
\end{align}
where $B_2$, $B_{p+\ve}$ were the optimal constants defined in Lemma \ref{lemma3}, $\Om_1,\Om_2$ are defined similarly as those in \eqref{eq-Omega}, and
\[
\xi_1:=\f{(\e(p-2))^{-1}\al_n B_2^2}p,\quad\xi_2:=\f{(\e\,\ve)^{-1}\al_n^{p-1+\ve}B_{p+\ve}^{p+\ve}}p\left(\f{2(1+\wt C)E(0)}\ell\right)^{\f{p+\ve-2}2}.
\]
Here we also used the inequality $\te^{p-2}\log\te<(\e(p-2))^{-1}$ for $0<\te<1$. Similarly, we have
\begin{equation}
\int_\Om\f{|\al_n\wt u_n|^p}{\al_n p^2}\,\rd x\le\f{\al_n^{p-1}}{p^2}\int_\Om{|\wt u_n|^p}\,\rd x\le\f{\al_n^{p-1}B_p^p}{p^2}\|\nb\wt u_n\|_2^p\le\xi_3\|\nb\wt u_n\|_2^2
\end{equation}
with
\[
\xi_3:=\f{\al_n^{p-1}B_p^p}{p^2}\left(\f{2(1+\wt C)E(0)}\ell\right)^{\f{p-2}2}.
\]

Next, we deal with the boundary term. Multiplying both sides of the fifth equation in \eqref{x13.36} by $\ze (\wt\vp+\wt\psi_n)$ with $\ze\in C_0^\infty(0,T)$ chosen in Step 3 and integrating over $\Ga_1\times(0,T)$, we have
\begin{align*}
& \int_0^T\ze\int_{\Ga_1}\wt u_n^{\,\prime}(\wt\vp+\wt\psi_n)\,\rd S\rd t+\int_0^T\ze\int_{\Ga_1}q(\wt\vp+\wt\psi_n)^2\,\rd S\rd t+\int_0^T\ze\int_{\Ga_1}r(\wt\vp^{\,\prime}+\wt\psi_n^{\,\prime})(\wt\vp+\wt\psi_n)\,\rd S\rd t=0.
\end{align*}
Performing the integration by parts for the first term yields
\[
\int_0^T\ze\int_{\Ga_1}\wt u_n^{\,\prime}(\wt\vp+\wt\psi_n)\,\rd S\rd t=-\int_0^T\ze\int_{\Ga_1}\wt u_n(\wt\vp^{\,\prime}+\wt\psi_n^{\,\prime})\,\rd S\rd t-\int_0^T\ze'\int_{\Ga_1}\wt u_n(\wt\vp+\wt\psi_n)\,\rd S\rd t.
\]
Then we have
\begin{align*}
\int_0^T\ze\int_{\Ga_1}q(\wt\vp+\wt\psi_n)^2\,\rd S\rd t & =-\int_0^T\ze\int_{\Ga_1}r(\wt\vp^{\,\prime}+\wt\psi_n^{\,\prime})(\wt\vp+\wt\psi_n)\,\rd S\rd t+\int_0^T\ze\int_{\Ga_1}\wt u_n(\wt\vp^{\,\prime}+\wt\psi_n^{\,\prime})\,\rd S\rd t\\
& \quad\,+\int_0^T\ze'\int_{\Ga_1}\wt u_n(\wt\vp+\wt\psi_n)\,\rd S\rd t.
\end{align*}
Owing to the choice of $\ze$ and the strict positivity of $q$, we deduce
\[
\int_\ep^{T-\ep}\int_{\Ga_1}(\wt\vp+\wt\psi_n)^2\,\rd S\rd t\longrightarrow0,\quad n\to\infty.
\]
Since $\ep>0$ was chosen arbitrarily small, we conclude that
\begin{equation}\label{3.97}
\int_0^T\!\!\!\int_{\Ga_1}(\wt\vp+\wt\psi_n)^2\,\rd S\rd t\longrightarrow0,\quad n\to\infty.
\end{equation}

Finally, we integrate \eqref{3.47} from $0$ to $T$ and take advantage of \eqref{3.50} and \eqref{3.92}--\eqref{3.97} to derive
\begin{align*}
\int_0^T\wt\cE_n(t)\,\rd t & \le\int_0^T\!\!\!\int_\Om\rho|\wt u_n^{\,\prime}|^2\,\rd x\rd t+\int_0^T\!\!\!\int_0^\infty g\int_\Om a|\nb\wt\eta_n^{\,t}|^2\,\rd x\rd s\rd t\\
& \quad\,+\int_0^T(1+\xi_1+\xi_2+\xi_3)\int_\Om|\nb\wt u_n|^2\,\rd x\rd t+\int_0^T\!\!\!\int_{\Ga_1}q(\wt\vp+\wt\psi_n)^2\,\rd S\rd t\longrightarrow0
\end{align*}
as $n\to\infty$. Thanks to the monotone decreasing of $\wt\cE_n$, we infer
\[
\int_0^T\wt\cE_n(t)\,\rd t\ge T\wt\cE_n(T).
\]
and eventually $\lim_{n\to\infty}\wt\cE_n(T)=0$, which contradicts with \eqref{3.51}. Hence we can conclude that \eqref{3.11} holds and so does \eqref{3.4}. The proof of Theorem \ref{lem1} is completed.\medskip

We close this article by mentioning that Corollary \ref{co1} can be proved analogously as above, so that we refrain from providing a detailed proof here. The interested readers are referred to \cite{c3} for further details.

\section*{Acknowledgements}

The authors thank the anonymous referees for valuable comments. The first author is supported by China Scholarship Council. The second author is supported by JSPS KAKENHI Grant Numbers JP22K13954, JP23KK0049 and Guangdong Basic and Applied Basic Research Foundation (No.\! 2025A1515012248).


\end{document}